\documentclass[a4paper12pt]{article}
\usepackage{amsfonts,amssymb,bm,empheq}
\usepackage{mathrsfs}
\usepackage{amsmath}
\usepackage{amsthm}
\usepackage{epsfig}
\usepackage[inline]{enumitem} 

\usepackage[dvipsnames]{xcolor}

 \newtheorem{theorem}{Theorem}[section]
 \newtheorem{fact}[theorem]{Fact}
 \newtheorem{lemma}[theorem]{Lemma}
 
\theoremstyle{definition}
 \newtheorem{definition}[theorem]{Definition}
 
\numberwithin{equation}{section}
\numberwithin{figure}{section}

\newcommand{\R}{\boldsymbol{R}}
\newcommand{\A}{\mathcal{A}}
\newcommand{\e}{\mathbf{e}}
\newcommand{\bb}{\mathbf{b}}
\newcommand{\n}{\mathbf{n}}
\newcommand{\hh}{\hat\kappa_2}
\newcommand{\hhh}{\hat\kappa_3}

\usepackage{graphicx}

\title{Geometry on curves which go around a Whitney umbrella}
\author{Masayuki Hara}
\date{\today}

\begin{document}
\maketitle


\begin{abstract}
We consider curves which go around Whitney umbrella.
Then we consider the geodesic and the normal curvatures,
ruled surfaces generated by the normal vector and
normal developable surfaces with respect to
the tangent and bi-tangent vectors of the curve.
Then we obtain several functions which represents
geometry on Whitney umbrella.
Looking at the top terms of them with respect to the radius,
we study geometry on Whitney umbrella.
\end{abstract}
\section{Introduction}
In this paper, we consider geometry on curves which
go around Whitney umbrella.
Whitney umbrella is known as a most frequently
appearing singularity on surfaces.
A map-germ $f:(\R^2,0)\to(\R^3,0)$ is a {\it Whitney umbrella}
if it is $\A$-equivalent to the map-germ $(u,v)\mapsto(u,v^2,uv)$ at the
origin, where two map-germs are said to be 
{\it $\A$-equivarent} 
if there exist diffeomorphism-germs $\varphi : (\R^2 ,0) \to (\R^2,0)$ and
$\psi : (\R^3,0) \to (\R^3,0)$ such that
$$g=\psi \circ f\circ \varphi^{-1} $$
holds.
It is also known that
unit normal vector field around a Whitney umbrella
cannot be smoothly extended beyond the singular point.
In \cite{FH}, it is shown that by taking the blow-up operation
on the source space of a Whitney umbrella, 
that is to say, taking the polar coordinate,
unit normal vector field
can be smoothly extended around the set which corresponds
to the singular point.
This fact strongly suggests that considering
a curve which goes around a Whitney umbrella,
then several geometric objects can be defined as
a direction-dependently.
In this paper, we consider
curves who has a constant radius with respect to
the polar coordinate around a Whitney umbrella.
Then we consider
the geodesic and the normal curvatures,
ruled surfaces generated by the normal vector and
normal developable surfaces with respect to
the tangent and bi-tangent vectors of the curve.
Then we obtain several functions which represents
geometry on Whitney umbrella depending on the angle
around the singular point, and the radius.
Looking at the top terms of them with respect to the radius,
we study geometry on Whitney umbrella.

\section{Preliminaries}

\subsection{A review of geometry on Whitney umbrella}
To see the differential geometric properties,
we use the following form to represent a Whitney umbrella.
\begin{theorem}[Bruce-West\cite{bw}, \cite{west}]\label{thm:bwnormal}
Let $g:(\bm{R}^2,0)\to (\bm{R}^3,0)$ be a Whitney umbrella.
Then, for any $ k \geq 3 $, there exist diffeomorphism-germ
 $\phi
 :(\bm{R}^2,0)\to (\bm{R}^2,0)$ and  $T\in SO(3)$
 such that
\begin{equation}\label{eq:bwnormal}
f_0=T\circ g\circ \phi  (u,v)
=\left(u,uv+B(v)+O(u,v)^{k+1},\sum_{j=2}^k
  A_j (u,v)+O(u,v)^{k+1}\right),
\end{equation}
   where
  $$B(v)=\sum_{i=3}^k \frac{b_i}{i!} v^i , A_j =\sum_{i=0}^j
  \frac{a_{i,j-i}}{i!(j-i)!} u^i v^{j-i} , a_{02} > 0 .$$
\end{theorem}
Here, $O(u,v)^{k+1}$ stands for a function whose order is
geater or equal to $k+1$.
We denote by $f_0$ the right-hand side of \eqref{eq:bwnormal},
and 
is called a {\it normal form of Whitnet umbrella}.
Since we use an element $T$ of $SO(3)$ as a diffeomorphism on the
target space, $f_0$
has the same difarential geometric properties
with $g$. Thus, all coefficients of $B(v)$ and $A_j(u,v)$ are difarential
gemetric invariants of given the Whitney umbrella $g$.
By using this form, geometry on Whitney umbrella
is studied by several authors, see \cite{bw,FH,hhnuy,sym,
hnuy,MN,tari}, for example.
Furtheremore, we have the following. 
\begin{lemma}\label{a02}
By a diffeomorphic transformation on the source space and
a scaling on the target space,
$f_0$ is transformed into the form \eqref{eq:bwnormal} with
$a_{02} = 1$. 
\end{lemma}
\begin{proof}
We divide the right-hand side of \eqref{eq:bwnormal} by $a_{02} \neq 0$, 
and take the coordinate transformation $u \mapsto a_{02} u$.
Then replacing $a_{20} a_{02} \mapsto a_{20}$, we have the assertion. 
\end{proof}
Following \cite[Section 2.3]{FH}, we consider 
the brow-up by a domain on a Whitney umbrella, 
we see the unit normal vector field near a Whitney umbrella 
can be continuously extended beyond the singular point.
Let $\mathcal{M}$ be the quotient space of
the product set
$\bm{R} \times S^1$ by the equivalence relation
$(r,\theta)\sim (-r,\theta +\pi)$. 
The map
$$\pi: \mathcal{M}\to \bm{R}^2 , 
[(r,\theta)]\mapsto (r\cos\theta
,r\sin\theta )$$
is called the {\it brow-up by a domain}, or {\it brow-up} for short. 
The set $\{r = 0\}$ corresponds to the origin.  

Although unit normal vector field near a Whitney umbrella 
cannot be continuously extended,
it is known that by this operation blow-up by a domain,
one can extend
a unit normal vector field including $\{r=0\}$ on $\mathcal{M}$:
\begin{theorem}[\cite{FH}]
Let $f_0$ be the right-hand side of
 \eqref{eq:bwnormal}, and let
 $\mathbf{n}$ be a unit normal vector of $f$ which is defined
 by outside of the origin. Then, 
$\tilde{\mathbf{n}}(r,\theta)=\mathbf{n}\circ \pi$
 can be smoothly extended on the set $\{r=0\}$ as
 \begin{equation}\label{unwu}
   \tilde{\mathbf{n}}(0,\theta)=\frac{1}{\mathcal{A}}
  (0,-a_{11}\cos\theta -a_{02} \sin \theta ,
  \cos \theta),
  \end{equation}
 where  $\mathcal{A}=\mathcal{A}(\theta)=\sqrt{\cos^2 \theta +(a_{11}
 \cos
  \theta +a_{02} \sin \theta )^2} (\neq 0)$. 
\end{theorem}
\subsection{Ruled surfaces}\label{sec:ruled}
Let $\gamma:(\R,0)\to(\R^3,0)$ be a curve,
and let $\delta:(\R,0)\to\R^3$ be a curve satisfying
$|\delta|=1$.
A surface $f(t,\beta)=\gamma(t)+\beta\delta(t)$ is called 
a {\it ruled surface}.
The curve $\gamma$ is called a {\it base curve}, and
$\delta$ a {\it director curve}.
The straight lines $t\mapsto \gamma(t)+\beta\delta(t)$
is called the {\it ruling}.
A ruled surface $f(t,\beta)=\gamma(t)+\beta\delta(t)$ is said to be 
{\it non-cylindrical} if $\delta'\ne0$ for any $t$.
If $f$ is a non-cylindrical ruled surface,
the surface
$\sigma(t)+\beta\delta(t)$ has the same image with $f$,
where
$\sigma=\gamma-(\gamma'\cdot\delta'/(\delta'\cdot\delta'))\delta$.
The curve $\sigma$ is called the {\it striction curve}
of $f$. See \cite[Section 14]{gray} for detail.
It is known that the singular values of a non-cylindrical ruled surface
are located on the striction curve (\cite[Lemma 2.2]{iztake}).
If a ruled surfaces has constant zero Gaussian curvature,
it is called a {\it developable surface}.
If $\delta'\equiv0$, then $f$ is called a {\it cylinder}.
If a non-cylindrical ruled surface satisfies $\sigma'\equiv0$,
where $\sigma$ is the striction curve,
then $f$ is called a {\it cone}.
Here $\equiv$ stands for the equality holds identically.
It is known that developable surface is locally classified into
cylinders, cones and tangent developable surfaces.
If $\delta$ is parallel to $\gamma'$, the ruled surface
$f(t,\beta)=\gamma(t)+\beta\delta(t)$ is called a
{\it tangent developable surface}.

\subsection{Height function along a curve and normal developable
  surfaces}
  \label{sec:nds}
Following \cite{izohflat}, we introduce 
the normal developable surfaces along a curve.
Let $\gamma:(\R,0)\to(\R^3,0)$ be a curve.
We assume that there exists a unit vector field
$\e$ along $\gamma$ such that $\gamma'=l\e$, where $l$ is a function.
Let $\n$ be a unit vector field along $\gamma$ which
satisfies $\e\cdot \n=0$.
We set $\bb=-\e\times \n$.
We define three functions $\kappa_1,\kappa_2,\kappa_3$ by the following
Frenet-Serre type formula:
   \begin{align}\label{eq:kappa123}
    \frac{d}{dt}
   \begin{pmatrix}
    \mathbf{e} \\
    \mathbf{b} \\
    \mathbf{n}
   \end{pmatrix}
    =
       \begin{pmatrix}
    0 & \kappa_1 & \kappa_2 \\
    - \kappa_1 & 0 & \kappa_3\\
    - \kappa_2 & - \kappa_3 & 0
   \end{pmatrix}
       \begin{pmatrix}
    \mathbf{e} \\
    \mathbf{b} \\
    \mathbf{n}
   \end{pmatrix}.
   \end{align}
We assume that $\n'\ne0$, namely
$(\kappa_2,\kappa_3)\ne(0,0)$ for any $t$.
Let us consider the function
$H(t,X) : (\R \times \R^3, 0) \to (\R,0)$ defined by
$$H(t , X) = \mathbf{n} \cdot (X - \gamma (t)).$$
Then by 
$H_t = (- \kappa_2 \mathbf{e} - \kappa_3
 \mathbf{b})(X - \gamma)$,
under the assumption $(\kappa_2,\kappa_3)\ne(0,0)$,
the discriminant set 
$$
{\cal D}_H =\{X\in\R^3\,|\,\text{there exists }t\text{ such that }
H(t,X) = H_t(t,X) = 0\}
$$
is parameterized by
$$
(t,\beta)\mapsto
\gamma + \frac{\beta}{\kappa_2} \sqrt{{\kappa_2}^2 + {\kappa_3}^2}
  \dfrac{(\kappa_3 \e - \kappa_2
  \bb)}{\sqrt{{\kappa_2}^2 + {\kappa_3}^2}},
$$
Setting $D=(\kappa_3 \e - \kappa_2\bb) / \sqrt{{\kappa_2}^2 + {\kappa_3}^2}$
and re-setting
$\beta \sqrt{{\kappa_2}^2 +{\kappa_3}^2} / \kappa_2\mapsto \beta$,
we obtain a ruled surface
$h(t, \beta) = \gamma(t) + \beta D(t)$.
By the construction, $h$ is a developable surface.
We set
\begin{equation}
    \delta = \kappa_1 (\hat{\kappa_2}^2 + \hat{\kappa_3}^2 ) - \hat{\kappa_2} '
    \hat{ \kappa_3 } + \hat{ \kappa_2} \hat{\kappa_3} ',
   \end{equation}
where 
$\hat{\kappa_i} = l \kappa_i$ $(i = 2, 3)$.
Calculating $D'$, we have
$$
D' =\frac{\delta( \hh \e + \hhh \bb)}{({\hh}^2 + {\hhh}^2)^{\frac{3}{2}}}.
$$
Thus $h$ is non-cylindrical if and only if
$\delta\ne0$ for any $t$.
Under the assumption $\delta\ne0$ for any $t$,
the striction curve $\sigma$ of $h$ is given by
$$\sigma = \gamma - \frac{l \hh}{\delta} (\hhh \e - \hh \bb).$$
By a direct calculation, we have
$$\sigma ' = \frac{k}{\delta^2} (\hhh \e - \hh \bb),$$
where
$$k = \delta (l \kappa_1 \hhh -2 l \hh ' - l' \hh ) + l \hh \delta' .$$
Thus $h$ is a cone if and only if
$\delta\ne0$ and $k=0$ for any $t$.

\section{Geometry on a curve which goes around a Whitney umbrella}
The unit normal vector has been defined by the blow-up, 
the geodesic and the normal curvatures at the origin can be defined
in a direction-dependent manner even at the origin.
In this section, we consider them.
\subsection{Geodesic curvature}
Let $f=f_0$ be the right-hand side of \eqref{eq:bwnormal}
and let 
  \begin{equation}\label{def:gamma}
   \gamma (\theta)= (u,v) = (r \cos \theta ,
  r \mathrm{sin} \theta)
  \end{equation}
 be a curve that $r$ is a constant. We consider the curve
\begin{align*}
  f\circ \gamma(\theta) &=f(r \cos \theta , r \sin \theta) \\
              & = \Big(r\cos \theta , r^2 \sin \theta
   \cos \theta + 
 \cdots , \\
& \hspace{1cm} \frac{a_{20}}{2} r^2 \cos^2 \theta + a_{11} r^2
   \sin \theta \cos \theta + \frac{a_{02}}{2}
 r^2 \sin^2 \theta + \cdots \Big).
  \end{align*}
We calculate the geodesic and the normal curvatures of the curve
$\theta\mapsto f \circ \gamma(\theta)$.
Which are functions of $(r,\theta)$, we set them
$\kappa_g(r,\theta)$ and $\kappa_n(r,\theta)$, respectively.
Since  the normal form is unique, the set 
$\{(r \mathrm{cos} \theta , r \mathrm{sin} \theta)\,|\,\theta\in[0,2\pi)\}$ 
has an absolute meaning.
\begin{definition}\label{def:firstterm}
Let $x(r,\theta)$ be a function of $(r,\theta)$.
If $x$ has the form
$$
x(r,\theta)=x_f(\theta)x_0(\theta)\,r^i+O(r^{i+1}),
$$
where $O(r^{i+1})$ stands for a function that is the order of $i+1$.
then the function $x_f(\theta)x_0(\theta)$ is called the
{\it first term} of $x$.
If $x_0(\theta)$ does not vanish for any $\theta$, then
the function $x_f(\theta)$ is also called the
{\it first term} of $x$.
\end{definition}

We have the following.
\begin{theorem}\label{kappag}
  If $\sin \theta \not= 0$, then the geodesic curvature $\kappa_g
 (r,\theta)$ is a $C^{\infty}$-function.
 Moreover, we have 
 \begin{align}
 \kappa_g(r,\theta) 
  = \frac{|\sin \theta|}{\mathcal{A}}g(s)
 + O(r),\label{eq:kgformula}
 \end{align}
 where
 \begin{equation}\label{eq:g}
  g=g(s)=
({a_{11}}^2 + 1) s^4 + a_{11}
 a_{02} s^3 + 3({a_{11}}^2 + 1) s^2 
+ a_{11}(4 a_{02} - a_{20}) s + a_{02} (a_{02} -
 a_{20}),
 \end{equation}
$s =\cos\theta/\sin\theta$.
\end{theorem}

\begin{proof}
 Each Christoffel simbol of $f(u,v)$ evaluated at
 $(r \cos \theta,r \sin \theta)$ is calicurated as 
\begin{align*}
 &\Gamma_{uu}^u (r \cos \theta , r \sin \theta)
 = \Gamma_{uv}^u (r \cos \theta , r \sin \theta)
 = \Gamma_{vv}^u (r \cos \theta , r \sin \theta)= O(1) , \\
 &\Gamma_{uu}^v (r \cos \theta , r \sin \theta)
 = \frac{1}{r (\mathcal{A}^2 + O(r) )} (a_{20}(a_{11} \cos \theta + a_{02} \sin
 \theta)) + O(1) , \\
 &\Gamma_{uv}^v (r \cos \theta , r \sin \theta)
 = \frac{1}{r (\mathcal{A}^2 + O(r) )} (\cos \theta  +
  a_{11} (a_{11}\cos \theta +
 a_{02} \sin \theta)) + O(1) , \\
 &\Gamma_{vv}^v (r \cos \theta , r \sin \theta)
 = \frac{1}{r (\mathcal{A}^2 + O(r) )} (a_{02} (a_{11} \cos \theta + a_{02}
 \sin \theta)) + O(1),
\end{align*}
where $O(1)$ stands for a function whose order is a constant.
 Since $\kappa_g$ is
 $$\kappa_g(r,\theta) = k_g \cdot \biggl(\tilde{\mathbf{n}} \times
 \frac{h_\theta}{|h_\theta|}\biggr) = \det \biggl( \tilde{\mathbf{n}}, 
 \frac{h_\theta}{|h_\theta|} , k_g
 \biggr),$$
 where 
  \begin{align*}
   k_g =\frac{1}{|h_\theta|^2} \sum_{\alpha \in \{ u,v \}} \Bigl\{
   \Bigl(
   \frac{\partial^2 \alpha}{\partial \theta^2} + \sum_{\zeta ,\eta \in \{ u,v
   \}} \Gamma^\alpha_{\zeta \eta} \frac{\partial \zeta}{\partial \theta}
   \frac{\partial \eta}{\partial \theta}
   \Bigr) h_{\alpha}\Bigr\},
  \end{align*}
we have the assertion.
\end{proof}

We set $\kappa_g(\theta) = \kappa_g(0,\theta)$, namely
the first term of $\kappa_g(r,\theta)$ is $\kappa_g(\theta)$.
By \eqref{eq:kgformula}, we remark that
if $\sin \theta \to 0$, then the geodesic curvature $\kappa_g(\theta) =
  \kappa_g(0,\theta)$ diverges at the order of $\theta^3$.
We have the following.
\begin{theorem}\label{thm:gsol}
The number of distinct real roots 
 of $\kappa_g(\theta) = 0$  is $2$, $1$ or $0$ on
 $ 0 < \theta < \pi$.
\end{theorem}
To show this theorem, since the geodesic curvature vanishes if and only if
$g(s)=0$ where $s=\cos\theta/\sin\theta$, and
the range of $\cos\theta/\sin\theta$ $(0<\theta<\pi)$
is the entire real number,
we show the number of the solution of
the quadratic equiation $g(s) = 0$ with respect to $s\in\R$
is $2$, $1$ or $0$.
\begin{lemma}\label{lem:gsol}
\begin{enumerate*}
\item\label{itm:g1} The equation $g(s) = 0$ 
never has quadruple solutions nor 
     two different $($real or complex$)$ double solutions.
 \item\label{itm:g2} The equation $g(s) = 0$ does not have triple solutions.
\end{enumerate*}
\end{lemma}
\begin{proof}
\ref{itm:g1}
By the general theory of the quadratic equiation,
$3 {a_{11}}^2 {a_{02}}^2 -8 ({a_{11}}^2 + 1)^2 \geq 0$
is a necessary condition that $g(s) = 0$ 
has quadruple solutions or 
two different real double solutions.
Thus we may assume $a_{11} \not= 0$.
By the general theory of the quadratic equiation again,
\begin{equation}\label{eq:gproof1}
{a_{11}}^2 (9 {a_{11}}^4 + 16 {a_{11}}^2 + 8) {a_{02}}^4 +
24({a_{11}}^2 + 1)^2 (3 {a_{11}}^2 +4) {a_{02}}^2 + 144 ({a_{11}}^2 +
1)^4 = 0
\end{equation}
is a necessary condition that $g(s) = 0$ 
has quadruple solutions or 
two different real double solutions.
We regard the equation \eqref{eq:gproof1} as
an equation with respect to $a_{02}$.
Since ${a_{02}}^2 > 0$, 
any of coefficients in the left-hand side of
\eqref{eq:gproof1} with respect to $a_{02}$ are positive,
there does not exist real number $a_{02}$
 such that \eqref{eq:gproof1} holds.
 Next we prove $g(s) = 0$ does not have two dfferent complex double solution.
 Since the similarity transformation of
 $\R^3$ does not change whether the geodesic curvature is $0$ or not,
 we assume $a_{02} = 1$ from the Lemma 
 \ref{a02}.
 By the general theory of quadratic equation,
the quadratic equiation $q(x)=0$, where
$
q(x)= a x^4 + b x^3 + c x^2 + d x + e
$
has a complex double solution, then
$\Delta = R = 0$ holds.
Here $\Delta$ is the discriminant of $q$, and
$R=b^3+8a^2d-4abc$.
For the case of the equiation $g(s)=0$, we have
 \begin{align}\label{eq:delta}
 \Delta =& -108 { a_{11} }^{10} (2 + a_{20})^2  
 - 27 { a_{11} }^8 (36 + 116 a_{20} + 13 {a_{20}}^2 + 2 {a_{20}}^3 + {a_{20}}^4)\nonumber \\
 & - 54 { a_{11} }^6 (-54 + 183 a_{20} - 7 {a_{20}}^2 + 11 {a_{20}}^3 +
 {a_{20}}^4) \nonumber \\
 & - 9 { a_{11} }^4 (-865 + 1340 a_{20} - 118 {a_{20}}^2 + 144
 {a_{20}}^3 + 3 {a_{20}}^4) \nonumber \\
 &- 36 { a_{11} }^2 (-125 + 125 a_{20} - {a_{20}}^2 + 28 {a_{20}}^3)
 - 16 (-1 + a_{20}) (5 + 4 a_{20})^2,
\end{align}
 and $R = {a_{11}}^3 - 12 a_{11} (1 + {a_{11}}^2)^2
 - 8 a_{11} (1 + {a_{11}}^2)^2 (-4 + a_{20})$.
 Let $\operatorname{Rem}_{a_{20}}(\Delta,R)$
 be the remainder of the polynomial $\Delta$
 with respect to $a_{20}$ devided by $R$, then 
$\operatorname{Rem}_{a_{20}}(\Delta,R)$ is
  $$- \frac{ (384 + 1232 {a_{11}}^2 + 1275 {a_{11}}^4 + 432 {a_{11}}^6)
 (240 + 848 {a_{11}}^2 + 
1113 {a_{11}}^4 + 648 {a_{11}}^6 + 144 {a_{11}}^8)^2}{4096({a_{11}}^2 + 1)}.$$
 Since $\operatorname{Rem}_{a_{20}}(\Delta,R) < 0$ for any $a_{11}$, 
there does not exist a pair of real numbers
 $(a_{20},a_{11})$ such that $\Delta = R = 0$ holds.  

\ref{itm:g2}
We assume $a_{02} = 1$.
By the general theory of the quadratic equiation again,
the quadratic equiation $q(x)=0$ above 
has a triple solution, then $\Delta = \Delta_0 = 0$,
where $\Delta_0=c^2-3bd+12ae$.
For the case of the equiation $g(s)=0$, we have
$\Delta_0 = 21 - 12 a_{20} +
 ( 18 - 9 a_{20} ) { a_{11} }^2 + 9 { a_{11} }^4 .$
Since
\begin{align*}
\label{eq:divd0}
 \operatorname{Rem}_{a_{20}}(\Delta,\Delta_0) &= - \frac{27}{ ( 4 + 3 {a_{11}}^2 )^4 }
 (128 + 537 { a_{11} }^2 + 939 { a_{11} }^4 + 818 { a_{11} }^6 + 354 {
 a_{11} }^8 \\
 &\hspace{70mm}
+ 72 { a_{11} }^{10} +   9 { a_{11} }^{12})^2,
\end{align*}
we see $\operatorname{Rem}_{a_{20}}(\Delta,\Delta_0)<0$
for any $a_{11}$.
Thus there does not exist 
a pair of real numbers $(a_{20},a_{11})$ such that
 $\Delta = \Delta_0 = 0$.
\end{proof}
Furthremore we have the following.
\begin{lemma}\label{thm:sol}
 There does not exist a triple of real numbers
 $(a_{20} , a_{11} , a_{02})$
 such that the number of distinct real roots 
 of $g(s) = 0$ is $4$ or $3$.
\end{lemma}
\begin{proof}
We assume $a_{02} = 1$ by the same reason as 
the proof of the Lemma \ref{lem:gsol} \ref{itm:g2}.
By the general theory of the quadratic equiation again,
the quadratic equiation $q(x)=0$ above,
the number of the distinct real roots 
is $4$ or $3$, then
$P = 8ac - 3b^2 < 0$. 
For the case of the equiation $g(s)=0$, we have
$$P = 24 { a_{11} }^4 + 45 { a_{11} }^2 + 24 > 0.$$ 
Thus, we have the assertion. 
\end{proof}
By Theorem \ref{thm:gsol},
the equation $g(s) = 0$ has 2, 1 or 0 solutions.
It is considered to be a special Whitney umbrella
when the equation $g(s) = 0$ has a double solution.
Here, we will discuss the geometric meaning of such a Whitney umbrella.
We set 
$$f(u,v) = \left(u , uv , 
\frac{1}{2} a_{20} u^2 + \frac{1}{2} a_{02}v^2\right)
$$
and the image of $f$ has a symmetry with respect to the $xy$-plane.
In this case, since $$g(s)= s^4 + 3s^2 + a_{02} (a_{02} - a_{20}),$$
$g(s) = 0$ has a real double root if and only if  $a_{20} = a_{02}$.
Then the double root is
$\cot \theta = 0$, namely $\theta = \pi/2$.
We set $(a_{20}, a_{11}, a_{02}) = (1,0,1)$ in the above map $f$.
Then this $f$ satisfies that the equation
$g(s) = 0$ has a real double root.
See Figure \ref{fig:a200211}.
\begin{figure}[htbp]
 \begin{center}
\includegraphics[width=50mm]{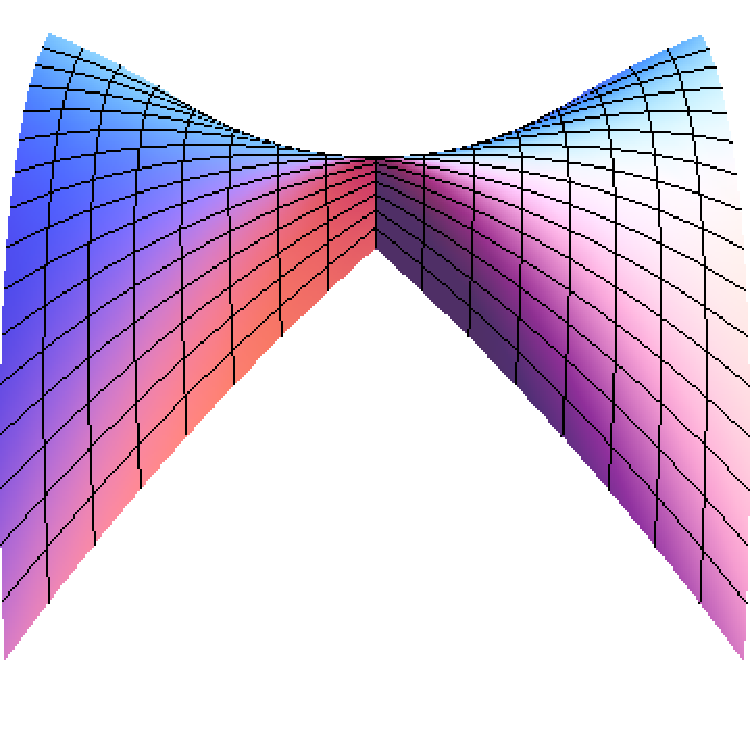}
 \end{center}
 \caption{The surface $f$ with $(a_{20}, a_{11}, a_{02}) = (1,0,1)$}
\label{fig:a200211}
\end{figure}
To see the geometric meaning of this condition, 
we introduce the Dupin
indicatrix of regular surface.

\begin{definition}
Let $f$ be a $f(u,v) = (u,v,au^2 + bv^2 + O(3) )$.
Concidering a curve on the intersection set $I$ between the  image of $f$
 and the plane $\{ z=z_0\}$ that is shifted the tangent plane at the
 origin by $z_0 >0$ in the direction of normal, then the second order
 term of defining equation of $I$ is a quadratic curve $au^2 + bv^2 = z_0$.
 We call this curve {\it the Dupin indicatrix at the origin}.
\end{definition}
We can define Dupin indicatrix of the Whitney umbrella similarly.
Let $f$ be a Whitney umbrella and it is the right-hand side of
$\eqref{eq:bwnormal}$.
Then $\partial v$ is a generator of the kernel 
$\operatorname{Ker} df_0$.
The vectors $f_u$ and $f_u \times f_{vv}$ are linearly independent,
and the plane 
$\langle f_u(0), f_u(0)
 \times f_{vv}(0) \rangle$ 
does not depend on the choice of the coordinate system
$(u,v)$.
We take $w \in \langle f_u(0), f_u(0)
\times f_{vv}(0) \rangle^{\bot} $. 

\begin{definition}
Considering a curve on the source space
which is 
the intersection set $I$ between the image of $f$ and
the plane $\{z = z_0\}$
that is shifted 
the plane $\langle f_u(0), f_u(0)\times f_{vv}(0) \rangle$ 
by $z_0(>0)$ in the direction of $w$,
then the second order terms of defining equation of $I$ is a quadratic
$$C : \frac{1}{2} a_{20} u^2 + a_{11} uv + \frac{1}{2} a_{02} v^2 = z_0.$$
We call $C$ the {\it 
Dupin indicatrix of the Whitney Umbrella on the source space}.
\end{definition}
Since we assumed $a_{11}=0$, the quadratic curve
$$C : \frac{1}{2} a_{20} u^2 + \frac{1}{2} a_{02} v^2 = z_0 $$
is an ellipse if $a_{20} > 0$,
a pair of parallel lines if $a_{20} = 0$, and
a hyperbola if $a_{20} < 0$. 
Especially, $a_{20} = a_{02}$ if and only if $C$ is a circle.
Moreover, we have the following well-known fact.
\begin{fact}
 The quadratic curve
$$C =\left\{(u,v)\in \R^2 \,\left|\,\frac{1}{2} a_{20} u^2 + a_{11} uv + \frac{1}{2} a_{02}
 v^2 = z_0  \right.\right\} 
$$
is a circle if and only if $a_{11} = 0$ and $a_{20} = a_{02}$.
\end{fact}

 \subsection{Normal curvature}
Let $f=f_0$ be the right-hand side of \eqref{eq:bwnormal},
and $\gamma$ be the curve defined in \eqref{def:gamma}.
Let $\kappa_n$ be the normal curvature of $\gamma$.
 We have the following.
 \begin{theorem}
 If  $\sin \theta \neq 0$, then the normal curvature $\kappa_n$ is
  continuous and
  $$\kappa_n(r,\theta) = \frac{\cos \theta}{\A\sin^2\theta} \bigl((3 a_{02} +
  a_{20})\sin^2 \theta +
  (a_{02} - a_{20}) \cos^2 \theta \bigr) + O(r).$$ 
 \end{theorem}
 We set $\kappa_n(\theta) = \kappa(0,\theta)$, namely the first term of
 $\kappa_n(r,\theta)$ is $\kappa_n(\theta)$.
We assume $a_{02} + a_{20}\neq 0$. Then
  $\kappa_n(\theta) = 0$
if and only if
$\cos2\theta = 2a_{02}/(a_{02}+a_{20})$ holds
$\theta \in (0,\pi)\setminus \{ \pi/2 \}$.
Moreover, if $a_{02} + a_{20} = 0$, there does not exist
$\theta \in (0,\pi)\setminus \{ \pi/2 \}$ such that
$\kappa_n(\theta) =0$ holds.

\subsection{Ruled surface defined by normal vector}
In this section, we take a curve which goes around a Whitney umbrella,
and consider
a ruled surface
whose director curve is the normal vector along the curve.
Let $f$ be a map defind by $f_0$ in \eqref{eq:bwnormal}.
Let 
$$\hat{\gamma} (\theta) = f \circ \gamma (\theta) = \left(r
\sin \theta , r^2 \sin \theta \cos \theta + \cdots , r^2 (a_{20} \cos^2
\theta + a_{11} \sin \theta \cos \theta + a_{02} \sin^2 \theta) +
\cdots\right)$$
be a curve.
If $r \neq 0$, then $f$ at $(r\cos\theta,r\sin\theta)$ is regular, 
there exists a unit normal vector
$\tilde{\n}$ along $\hat\gamma$
which can be regarded as function of $\theta$.
We have the ruled surface 
$g(\theta,\beta) = \hat{\gamma}(\theta) + \beta \tilde{\mathbf{n}}(\theta)$.
Calculating Gaussian curvature $K$ of the ruled surface $g$, we have
$$
K=
-\dfrac{\sin\theta^2}
{\beta^2}+O(r).
$$
Thus we see that if $\sin \theta \neq 0$, then 
the first term of $K \neq 0$ holds. 
The striction curve $\sigma$ of $g$ is given by
$\sigma = \hat{\gamma} + \beta(r,\theta) \tilde{\mathbf{n}},$
where
\begin{equation}\label{eq:beta}
\beta(r,\theta) = \frac{2 \cos \theta \left(a_{02} \cos^2 \theta
+ (2 a_{02} + a_{20}) \sin^2 \theta\right) \A^3}{{a_{02}}^2} r^2 +
O(r^3).
\end{equation}
Therefore, the first term of $\beta$ is 
$2 \cos \theta \left(a_{02} \cos^2 \theta
+ (2 a_{02} + a_{20}) \sin^2 \theta\right)$.
If $a_{02}+a_{20} \neq 0$, then 
the first term of $\beta$ vanishes if and only if
$\cos2\theta = (3a_{02}+a_{20})/(a_{02} + a_{20})$
or $\theta=\pi/2$.

\section{Normal developable surface on a curve which goes around a
 Whitney umbrella} 
In this section, 
we consider the normal developable surface defined in 
Secion \ref{sec:nds}.
Let $f=f_0$ be the right-hand side of \eqref{eq:bwnormal}
Let us consider a circle
$\hat{\gamma}(\theta)=f(r\cos\theta,r\sin\theta)$
for small $r>0$ as in \eqref{def:gamma}.
Since $r>0$, there exist a unit vector $\e(\theta)$ such
that $\hat{\gamma}_\theta = l(\theta) \e(\theta)$,
where $l$ is a function.
Let $\n$ be a unit vector along $\hat\gamma$ satisfying
$\e\cdot\n=0$.
We set  $\bb = - \e \times \n$.
As we saw in Section \ref{sec:nds},
we obtain the normal developable surface with respect to
$\{\e,\n,\bb\}$,
and functions $\kappa_1,\kappa_2,\kappa_3$ by the equation \eqref{eq:kappa123}.
Furthermore, we obtain functions 
$\delta$, $k$ which describe the
normal developable surface
is a cylinder or a cone respectively.
These are depend on the choices of $\n$.
In this section, 
we study asymptotic behavior of
these functions when $r\to0$, by looking at the
first terms of them with respect to $r$,
in the cases of $\n$ is the normal vector $\tilde\n$ of Whitney umbrella,
and $\n=-\e\times\tilde\n$. 

If $\n = \tilde{\n}$, the functions $\kappa_1,\kappa_2,\kappa_3$
can be calculated as follows:
  \begin{align*}
   \kappa_1 = \frac{\hat\kappa_1}{l^2},\quad
   \kappa_2 = \frac{\hh}{l},\quad
   \kappa_3 = \frac{\hhh}{l},
  \end{align*}
where $\hat\kappa_i$ are $C^{\infty}$-functions
defined by
   \begin{align*}
   \hat\kappa_1 &= \frac{F_{\hat\kappa_1}}{ \A} r^3 + O(r^4),\\ 
   \hh &=  \frac{F_{\hh}}{\A} r^2 + O(r^3), \\
   \hhh &= - \frac{F_{\hhh}}
    { (1 + {a_{11}}^2) \cos^2 \theta
     + a_{02} \sin \theta (2 a_{11} \cos \theta + a_{02} \sin \theta)} r
      + O(r^2), \\
   F_{\hat\kappa_1}(\theta) &=  (1 + {a_{11}}^2) \cos^4 \theta
    +  a_{02} a_{11} \cos^3\theta \sin\theta  
    + 3 (1 + {a_{11}}^2) \cos^2 \theta \sin^2 \theta \\
    &\hspace{1.5cm} +  a_{11} (4 a_{02} - a_{20} ) \cos \theta \sin^3 \theta
     +  a_{02} (a_{02} - a_{20}) \sin^4 \theta,\\
   F_{\hh}(\theta)&= \cos\theta\left( a_{02} \cos^2 \theta + (2 a_{02} +
    a_{20})
    \sin^2\theta\right),\\
   F_{\hhh}(\theta)&= a_{02} \sin \theta.
   \end{align*}
Thus $F_{\hat\kappa_i}(\theta)$ $(i=1,2,3)$ 
are the first terms of $\hat\kappa_i$.
By using them and $l = r \sin \theta + O(r^2)$, we have
  \begin{align*}
   \delta &= \frac{4 a_{02} \sin^2 \theta }{\A^5}F_{\delta}(\theta)
   r^3 + O(r^4),\\
F_{\delta}(\theta)&= (1 + {a_{11}}^2) (2 a_{02} + a_{20}) \cos^4 \theta
   +2 {a_{02}} a_{11} (2 a_{02} + a_{20}) \cos^3 \theta \sin \theta \\
   &+   ({a_{02}}^2 (2 a_{02} + a_{20}) - (1 + {a_{11}}^2)
   (a_{02} + 2 a_{20})) \cos^2
   \theta \sin^2 \theta \\
   &-2 {a_{02}} a_{11} (a_{02} + 2 a_{20}) \cos \theta \sin^3
   \theta \\
   &- {a_{02}}^2 (a_{02} + 2 a_{20}) \sin^4 \theta
  \end{align*}
and
  \begin{align*}
   k = & -\frac{12 a_{02}(a_{02} - a_{20})\sin^4 \theta}{ \left(1 + {a_{02}}^2
   + {a_{11}}^2 + (1 - {a_{02}}^2 + {a_{11}}^2) \cos 2\theta + 
   2 a_{02} a_{11} \sin 2\theta\right)^2} F_k(\theta)r^6
+ O(r^7),\\
F_k(\theta) =& (4 a_{02} + a_{20}) \cos^4\theta + (a_{02} - a_{20}) 
\cos^2\theta\sin^2\theta + (a_{02} + 2 a_{20})\sin^4\theta.
  \end{align*}
Thus  $F_{\delta}(\theta)$ and $F_k(\theta)$ are
the first terms of $\delta$ and $k$ respectively.
We have the following.
\begin{theorem}
{\rm (1)}
If $2a_{20} + a_{02} < 0$, then the number of 
$\theta\in(0,\pi)\setminus\{\pi/2\}$ satisfying
that $F_{\hh}(\theta)=0$ is generically $2$ or $0$.
{\rm (2)}
If $2 a_{02} + a_{20} \neq 0$, then the number
of $\theta\in(0,\pi)$ satisfying that 
$F_{\delta}(\theta) =0$ is generically 
   $4$,
    $2$, or $0$.  
If $2 a_{02} + a_{20} = 0$, then
$F_{\delta} \neq 0$ for any $\theta\in (0,\pi)$.
\end{theorem}
\begin{proof}
(1) is obtained by the formula of
$F_{\hh}(\theta)$.
We show (2).
The equation $F_{\delta}(\theta)=0$ is a quartic equation 
with respect to $\cos\theta/\sin\theta$,
we have the first assertion of (2).
If $2 a_{02} + a_{20} = 0$, then
$$
F_{\delta}(\theta)
=
 (1 + {a_{11}}^2) \cos^2 \theta + 2 a_{02} a_{11} \cos \theta \sin
   \theta  + {a_{02}}^2 \sin^2 \theta
= \cos^2 \theta + (a_{11} \cos \theta + a_{02} \sin \theta)^2,
$$
we have the assertion.
\end{proof}
Let us set $(a_{20}, a_{11}, a_{02}) = (-2,0,1)$
in the above $f$. Then $F_{\hh}(\theta)$ does not vanish
for any $\theta$. See Figure \ref{fig:2a0220}.
\begin{figure}[htbp]
 \begin{center}
\includegraphics[width=50mm]{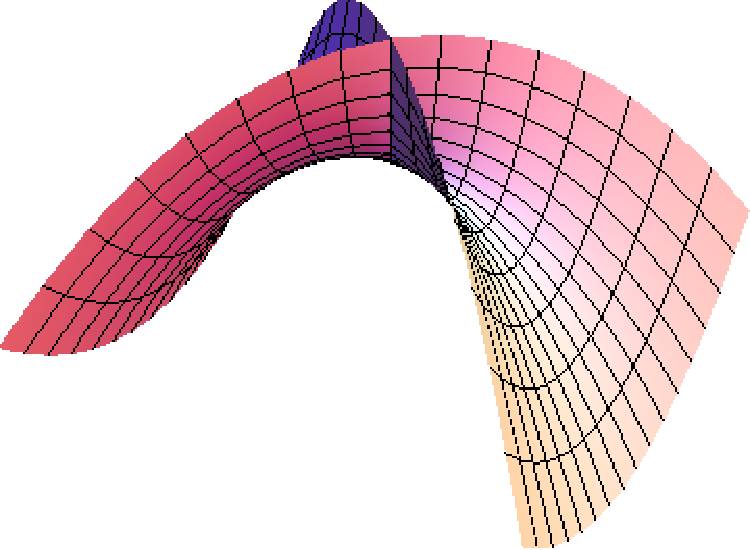}
 \end{center}
 \caption{The surface $f$ with $(a_{20}, a_{11}, a_{02}) = (-2,0,1)$}
\label{fig:2a0220}
\end{figure}
Comparing \eqref{eq:beta} and $F_{\hh}$,
we have the following theorem.
\begin{theorem}
The condition
  $F_{\hh} = 0$ holds if and only if the first term of $\beta(\theta)$
  vanishes, where
  $\beta(\theta)$ is given by \eqref{eq:beta}.
 \end{theorem}
  Moreover, For the function $k$, 
  we have the following.
   \begin{theorem}\label{thm:4a0220}
We assume  $a_{02} \neq a_{20}$ and
    $4 a_{02} + a_{20} \neq 0$.
Setting $s=\cos\theta/\sin\theta$, 
then real roots of $F_k(\theta) = 0$ satisfy
   \begin{equation*}
    s =
     \begin{dcases*}
      \pm \sqrt{\frac{a_{20} - a_{02} + \sqrt{T} }{4 a_{02} + a_{20}} }
      & {\rm if} $- 4 a_{02} < a_{20} < - 1 / 2 a_{02}$, \\
      \pm \sqrt{\frac{a_{20} - a_{02} \pm \sqrt{T} }{4 a_{02} + a_{20}} }
      &{\rm if} $ - 5 a_{02} \leq a_{20} < - 4 a_{02} $, \\
      0 \,\,\text{\rm (double root)} 
      &{\rm if} $a_{20} = - 1/2 a_{02}$, \\
     \text{\rm No real root.} & {\rm otherwise,}
     \end{dcases*}
   \end{equation*}
    where $T = - (3 a_{02} + 7 a_{20})(5 a_{02} + a_{20})$.
    Moreover, if $ 4 a_{02} + a_{20} = 0$, then real roots of $F_k(\theta)
    = 0$ are $\theta = \pm \cot^{-1}( \sqrt{7/5})$. 
   \end{theorem}
    \begin{proof}
     Let $p$ be a polynomial
     $$p = p(s) = (4 a_{02} + a_{20}) s^4 + (a_{02} - a_{20}) s^2
 + (a_{02} + 2 a_{20})$$
which is obtained by setting $s=\cos\theta/\sin\theta$ in $F_k(\theta)$.
    If $4 a_{02} + a_{20} = 0$, then $p(s) = 5 a_{02}  s^2
     - 7 a_{02} $. Thus, we have the last assertion.
          If $a_{02} \neq a_{20}$ and
    $4 a_{02} + a_{20} \neq 0$,
   the necessary condition that $p(s) = 0$ has real root is $T \geq
   0$. It is equivarent $- 5 a_{02} \leq a_{20} \leq -3/7 a_{02}$.
   Then,
   \begin{equation}\label{eq:s^2}
     s^2 = \frac{a_{20} - a_{02} \pm \sqrt{T} }{4 a_{02} + a_{20}} .
   \end{equation}
     Thus, 
the condition that 
the right-hand side of \eqref{eq:s^2} is not negative
is also a necessary condition for $p(s) = 0$ has a real root.
    If $4 a_{02} + a_{20} > 0$, 
then 
     $$\frac{a_{20} - a_{02} - \sqrt{T} }{4 a_{02} + a_{20}} < 0$$ holds,
     and
    $(4 a_{02} + a_{20})(a_{02} + 2 a_{20}) < 0 $
    is equivarent to
$$\frac{a_{20} - a_{02} + \sqrt{T} }{4 a_{02} +
a_{20}} > 0.$$
    This condition implies $-4 a_{02} < a_{20} <-1/2 a_{02}
    $. 
Since $a_{20} > -4 a_{02}$, if $-4 a_{02} < a_{20} < -1/2 a_{02}
    $, then $p(s) = 0$ has real roots. 
Moreover, if $4 a_{02} + a_{20} < 0$, 
then
    $$\frac{a_{20} - a_{02} - \sqrt{T} }{4 a_{02} + a_{20}} > 0$$
holds,
 and
    $(4 a_{02} + a_{20})(a_{02} + 2 a_{20}) > 0 $
    is equivarent to
    $$\frac{a_{20} - a_{02} + \sqrt{T} }{4 a_{02} + a_{20}} > 0.$$
The condition
$(4 a_{02} + a_{20})(a_{02} + 2 a_{20}) > 0 $
is equivalent to $ a_{20}  < -4 a_{02} , \, -1/2 a_{02} < a_{20}$.
    Since $a_{20} < -4 a_{02}$ and $- 5 a_{02} \leq a_{20} \leq -3/7
    a_{02}$, 
    if $- 5 a_{02} \leq a_{20} < -4 a_{02} $, then $p(s) = 0$ has real
    roots. Finally, if $a_{02} + 2a_{20} = 0$, then $p(s) = s^2(7/2 a_{02}s^2
     +3/2 a_{02})$. Thus, $p(s) = 0$ has real root $s=0$.
    From the avobe, we have the assertion.    
    \end{proof}
We set $(a_{20}, a_{11}, a_{02}) = (-4,0,1)$ in the above $f$.
As in Theorem \ref{thm:4a0220},
the real roots of $F_k(\theta) = 0$ 
can be calculated as $\theta = \pm \cot^{-1}( \sqrt{7/5})$
concretely. See Figure \ref{fig:4a02a20}.
\begin{figure}[htbp]
 \begin{center}
\includegraphics[width=50mm]{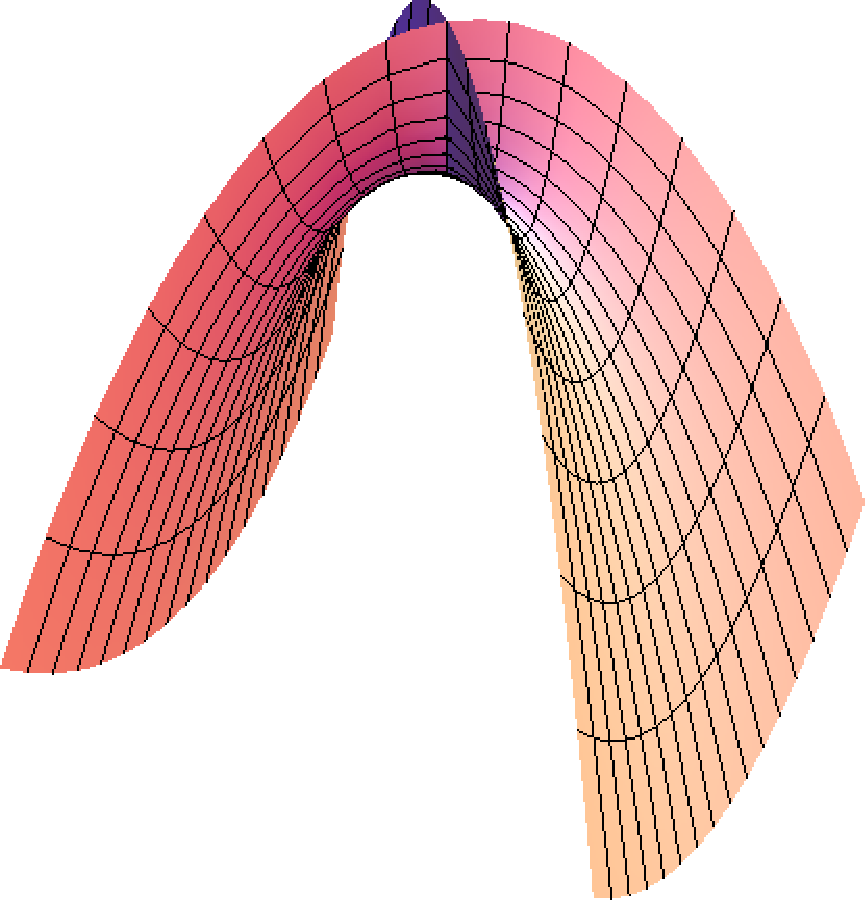}
 \end{center}
 \caption{The surface $f$ with $(a_{20}, a_{11}, a_{02}) = (-4,0,1)$}
\label{fig:4a02a20}
\end{figure}

In the case of $\n = - \e \times \tilde{\n}$, 
it holds that 
  \begin{align*}
   \kappa_1 &= \frac{{\hat{\kappa}_1}}{l} ,\\
   \kappa_2 &= \frac{1}{ l^2 \A}\bigl( (1 + {a_{11}}^2) \cos^4 \theta
   +  a_{02} a_{11} \cos^3\theta \sin\theta
   + 3 (1 + {a_{11}}^2) \cos^2 \theta \sin^2 \theta \\
   &\hspace{1.5cm} +  a_{11} (4 a_{02} - a_{20} ) \cos \theta \sin^3 \theta
   +  a_{02} (a_{02} - a_{20}) \sin^4 \theta
   \bigr) r^3 + O(r^4)\\
   &= \frac{\sin^4 \theta}{ l^2 \A} g(\cos\theta/\sin\theta) r^3 + O(r^4), \\
   \kappa_3 &= \frac{\hhh}{l},
  \end{align*}
  where $g$ is the polynomial 
given by
\eqref{eq:g} and $\hat{\kappa}_1$, $\hhh$ are $C^{\infty}$-function that satisfy
  \begin{align*}
   \hat{\kappa}_1 &= - \frac{ \cos\theta\left( a_{02} \cos^2 \theta
   + (2 a_{02} + a_{20})\sin^2
   \theta\right)}{\A} r^2 + O(r^3), \\
   \hhh &= - \frac{a_{02} \sin \theta}
   { (1 + {a_{11}}^2) \cos^2 \theta
   + a_{02} \sin \theta (2 a_{11} \cos \theta + a_{02} \sin \theta)} r
   + O(r^2).
  \end{align*}
By the formula of $\kappa_2$,
we have the following.
  \begin{theorem}
   The first term of $\kappa_2$ is vanishes if and only if
   $g(\cos\theta/\sin\theta) = 0$ holds, where
$g$ is given by \eqref{eq:g}.
  \end{theorem}
  Moreover, for $\delta$ and $k$, we have
  \begin{align}
   \delta &= \frac{a_{02}}{ \sin^2 \theta \A^5}\label{eq:delta2}
   \Bigl(
   3 (1 + {a_{11}}^2)^2 \cos^7 \theta
   + 7 a_{02} a_{11} (1 + {a_{11}}^2) \cos^6 \theta \sin \theta \\
   &\hspace{2cm}+ \bigl(8 (1 + {a_{11}}^2)^2 + {a_{02}}^2 (2 + 5 {a_{11}}^2)\bigr) \cos^5
   \theta \sin^2 \theta \nonumber\\
   &\hspace{2cm}+  a_{02} a_{11} ({a_{02}}^2 + 16 (1 + {a_{11}}^2))
   \cos^4 \theta \sin^3 \theta \nonumber\\
   &\hspace{2cm}+ \bigl(9 (1 + {a_{11}}^2)^2 + {a_{02}}^2 (3 + 6 {a_{11}}^2)
   + a_{02} (a_{20} + 2 {a_{11}}^2 a_{20})\bigr)
   \cos^3 \theta \sin^4 \theta \nonumber\\
   &\hspace{2cm}
   + a_{11} \bigl(-4 {a_{02}}^2 (a_{02} - a_{20}) + (1 + {a_{11}}^2)( 23 a_{02}
   - 
   2  a_{20} )\bigr) \cos^2 \theta \sin^5 \theta \nonumber\\
   &\hspace{2cm}- a_{02} \bigl(2 {a_{02}}^3 - a_{02} (7 + 19 {a_{11}}^2)
   + a_{20} (1 - 2 {a_{02}}^2  + 
   4 {a_{11}}^2 )\bigr) \cos \theta \sin^6 \theta \nonumber\\
   &\hspace{2cm}+ {a_{02}}^2 a_{11} (5 a_{02} - 2 a_{20}) \sin^7 \theta
   \Bigr) r^2 + O(r^3), \nonumber\\
       k &= - \frac{64 a_{02}}{ \sin^4 \theta \A^4}\label{eq:k2}
   \Bigl(\sum_{i+j=14} k_{i,j} \cos^{i}\theta\sin^j\theta 
   \Bigr)r^4 + O(r^5),
  \end{align}
   where
   \begin{align*}
    k_{14,0} &= -3 (1 + {a_{11}}^2)^4,\\
    k_{13,1} &= -8 a_{02} a_{11} (1 + {a_{11}}^2)^3,\\
    k_{12,2} &= - (1 + {a_{11}}^2)^2
                 (9 (1 + {a_{11}}^2)^2 + {a_{02}}^2 {a_{11}}^2),\\
    k_{11,3} &=   a_{11} (1 + {a_{11}}^2) \left( {a_{02}}^3 (7 + 16
                    {a_{11}}^2) 
               + (1 + {a_{11}}^2)^2 (31 a_{02} - 15 a_{20})\right), \\
    k_{10,4} &=   {a_{02}}^4 {a_{11}}^2 (16 + 19 {a_{11}}^2)
             + (1 + {a_{11}}^2)^2 (-69 (1 + {a_{11}}^2)^2\\
             &        \, \, \,\hspace{2em}+ {a_{02}}^2 (46 + 265 {a_{11}}^2) - 
                7 a_{02} (3 + 14 {a_{11}}^2) a_{20}), \\
    k_{9,5} &= - a_{11} \bigl(-{a_{02}}^5 (3 + 8 {a_{11}}^2) - {a_{02}}^3
                (288 + 869 {a_{11}}^2 + 581 {a_{11}}^4) \\
               &        \, \, \,\hspace{2em}+ 
               {a_{02}}^2 (147 + 400 {a_{11}}^2 + 253 {a_{11}}^4) a_{20} + 
                18 (1 + {a_{11}}^2)^3 (11 a_{02} + a_{20})\bigr),
      \end{align*}
     \begin{align*}
    k_{8,6} &= {a_{02}}^3 ({a_{02}}^3 {a_{11}}^2 + 
           a_{02} (64 + 586 {a_{11}}^2 + 609 {a_{11}}^4)
           - (54 + 353 {a_{11}}^2 + 
           332 {a_{11}}^4) a_{20}) \\
           &\hspace{3em}
                 + (1 + {a_{11}}^2)^2 ({a_{02}}^2 (28 - 25 {a_{11}}^2) - 
             147 (1 + {a_{11}}^2)^2 - a_{02} (45 + 152 {a_{11}}^2)
             a_{20}), \\
    k_{7,7} &= - a_{11} \bigl(-{a_{02}}^5 (188 + 337 {a_{11}}^2) - 
           4 {a_{02}}^3 (91 + 216 {a_{11}}^2 + 125 {a_{11}}^4)
             - 21 (1 + {a_{11}}^2)^3 a_{20} \\
           &\hspace{3em}+ 
          {a_{02}}^4 (149 + 233 {a_{11}}^2) a_{20} + 
             {a_{02}}^2 (311 + 737 {a_{11}}^2 + 426 {a_{11}}^4) a_{20} \\
             &\hspace{5em}+ 
          a_{02} (1 + {a_{11}}^2) (621 (1 + {a_{11}}^2)^2 + 2 {a_{11}}^2
          {a_{20}}^2)\bigr), \\
    k_{6,8} &= {a_{02}}^2 \Bigl(
                 {a_{02}}^2 \bigl(121 + 844 {a_{11}}^2 + 697 {a_{11}}^4 + 
                  {a_{02}}^2 (24 + 95 {a_{11}}^2)\bigr) \\
               &\hspace{3em} - 
                a_{02} \bigl(135 + 687 {a_{11}}^2 + 528 {a_{11}}^4
               + {a_{02}}^2 (24 + 82 {a_{11}}^2)\bigr) a20 - 
                      2 {a_{11}}^2 (3 + 8 {a_{11}}^2) {a_{20}}^2 \\
                  &\hspace{5em}+ 
                  3 (1 + {a_{11}}^2)^2 (-36 (1 + {a_{11}}^2)^2
                       - 35 {a_{02}}^2 (1 + 9 {a_{11}}^2) + 
                     2 a_{02} (-4 + {a_{11}}^2) a_{20})\Bigr), \\
     k_{5,9} &= - a_{11} \Bigl(
                  -11 {a_{02}}^7 - 2 {a_{02}}^5 (161 + 177 {a_{11}}^2)
                        + 11 {a_{02}}^6 a_{20} \\
              &\hspace{1em}+ 
               {a_{02}}^4 (289 + 262 {a_{11}}^2) a_{20} 
                + 
                    {a_{02}}^3 \bigl(70 + 485 {a_{11}}^4 + 6 {a_{20}}^2 
                    + {a_{11}}^2 (555 + 44 {a_{20}}^2)\bigr) \\
                        &\hspace{3em}+
                  (1 +  {a_{11}}^2)
                      (24 (1 + {a_{11}}^2)^2 (23 a_{02} - 2 a_{20}) \\
        &\, \, \,\hspace{2em}- 
                        a_{02} (3 + 10 {a_{11}}^2) {a_{20}}^2 + 
                          {a_{02}}^2 (218 + 221 {a_{11}}^2) a_{20})\Bigr), \\
      \end{align*}
     \begin{align*}
    k_{4,10} &= {a_{02}}^2 \bigl({a_{02}}^4 (46 + 29 {a_{11}}^2) + 
                {a_{02}}^2 (76 + 425 {a_{11}}^2 + 233 {a_{11}}^4)\bigr) \\
           &\hspace{1em}+ 
           a_{02} \bigl(4 {a_{02}}^4 (-11 + 6 {a_{11}}^2) - 
           {a_{02}}^2 (131 + 660 {a_{11}}^2 + 520 {a_{11}}^4)\bigr) a_{20} \\
            &\hspace{1em}+ {a_{20}}^2\bigl(-2 {a_{02}}^4 (1 + 
            28 {a_{11}}^2) + {a_{02}}^2 (1 + 34 {a_{11}}^2 + 38
             {a_{11}}^4)\bigr)\\
 &        \, \, \,\hspace{2em}
             - 3 (1 + {a_{11}}^2)^2 \bigl(45 + 2 {a_{11}}^2 (158 +
                     {a_{20}}^2)\bigr),\\
   k_{3,11} &= - a_{02} a_{11} \bigl(32 {a_{02}}^6 - 66 {a_{02}}^5 a_{20} - 
           6 a_{02} (1 + {a_{11}}^2) (19 + 64 {a_{11}}^2) a_{20} \\
            &\hspace{1em}+ 
              {a_{02}}^3 (366 + 485 {a_{11}}^2) a_{20}
            + 6 (1 + 5 {a_{11}}^2 + 4 {a_{11}}^4) {a_{20}}^2 \\
           &\hspace{2em}+ 
              {a_{02}}^4 (-322 - 385 {a_{11}}^2 + 34 {a_{20}}^2) + 
            {a_{02}}^2 (72 (1 + {a_{11}}^2) (7 + 18 {a_{11}}^2)
            - (29 + 52 {a_{11}}^2) {a_{20}}^2)\bigr),  \\
   k_{2,12} &= - {a_{02}}^2 \Bigl({a_{02}}^2 \bigl(50 + 8 {a_{02}}^4
             + 614 {a_{11}}^2 + 804 {a_{11}}^4 - 
              {a_{02}}^2 (62 + 157 {a_{11}}^2)\bigr) \\
                 &\hspace{3em}+ 
              2 a_{02} \bigl(1 - 8 {a_{02}}^4 - 107 {a_{11}}^2 - 168 {a_{11}}^4 + 
              {a_{02}}^2 (35 + 97 {a_{11}}^2)\bigr) a_{20} \\
               &\hspace{5em}+ \bigl(2 + 8 {a_{02}}^4 + 23 {a_{11}}^2 
              + 
              36 {a_{11}}^4 - 4 {a_{02}}^2 (2 + 7 {a_{11}}^2)\bigr) {a_{20}}^2
                        \Bigr), \\
   k_{1,13} &=  {a_{02}}^3 a_{11} \bigl(17 {a_{02}}^4 - 19 {a_{02}}^3 a_{20}
             + 2 a_{02} (7 + 72 {a_{11}}^2)
                 a_{20} \\
                &\hspace{3em}- 
                  4 (1 + 6 {a_{11}}^2) {a_{20}}^2 + 2 {a_{02}}^2
                  (-32 - 132 {a_{11}}^2 + {a_{20}}^2)\bigr), \\
   k_{0,14} &= -64 {a_{02}}^4 ({a_{02}}^2 
(-7 + 2 {a_{02}}^2 + 36 {a_{11}}^2) \\
&        \, \, \,\hspace{2em}- 
     4 a_{02} (-2 + {a_{02}}^2 + 6 {a_{11}}^2) a_{20}
     + (-1 + 2 {a_{02}}^2 
                + 
                   6 {a_{11}}^2) {a_{20}}^2).
 \end{align*}
We have the following theorem.
   \begin{theorem}
If $\n = - \e \times \tilde{\n}$,
   then the number of $\theta \in (0, \pi)$ that
    \begin{enumerate}
     \item  the
   first term of   
$\delta$ vanishes at $\theta$ is generically $7$,
   $5$, $3$, or $1$,
     \item
the first term of $k$ vanishes at $\theta$ is generically $14$,
    $12$, $10$, $8$, $6$, $4$, $2$ or $0$.	     
    \end{enumerate}
    \end{theorem}
\begin{proof}
By $1 + a_{11}^2 \neq 0$ and the formulas \eqref{eq:delta2} and \eqref{eq:k2},
we have the assertion.
\end{proof}

\section{Declarations}
\subsection{Funding and/or Conflicts of interests/Competing interests}
The author has no relevant financial or non-financial interests to
disclose and the author has no competing interests
to declare that are relevant to the content of this article.

\end{document}